\documentclass[10pt]{amsart}
\usepackage[letterpaper, margin=1.5in]{geometry}
\usepackage{graphicx}
\usepackage{amsmath,amssymb,amsthm,mathtools}
\usepackage[all]{xy}
\usepackage[utf8]{inputenc}
\usepackage[hidelinks]{hyperref}
\usepackage{enumitem}

\newtheorem{theorem}{Theorem}[section]

\newtheorem{remark}[theorem]{Remark}

\newtheorem{lemma}[theorem]{Lemma}
\newtheorem{proposition}[theorem]{Proposition}

\newtheorem*{theorem*}{Theorem}
\newtheorem*{question*}{Question}
\begin{document}

\title{Some properties of a Brauer class}
\date{\today}

\author{Qixiao Ma}
\address{Department of Mathematics, Columbia University,}
\email{qxma10@math.columbia.edu}

\begin{abstract}
Let $X$ be a smooth proper curve defined over a field $k$. The representability of the relative Picard functor is obstructed by a class $\alpha\in\mathrm{Br}(\mathrm{Pic}_{X/k})$. We show the associated division algebra on $\mathrm{Pic}^0_{X/k}$ has natural involutions. We show the class $\alpha$ splits at some height one points in $\mathrm{Pic}_{X/k}$. \end{abstract}
\maketitle

\tableofcontents

\section{Introduction}
Let $k$ be a field, let $X$ be a projective variety defined over $k$. Let $M$ be the moduli space of stable sheaves on $X$ with fixed rank and chern classes. The existence of tautological sheaves on $X\times_k M$ is obstructed by a Brauer class $\alpha\in\mathrm{Br}(M)$, see \cite[I.3.3]{caldararu}.

Brauer classes over fields are uniquely represented by central division algebras. Algebraists are still working on their structure theory, and asked if certain examples exist, see \cite{AESU}. It is interesting to check if the aforementioned Brauer classes, when restricted to the generic points of $M$, would provide some useful examples.

The coarsest invariants of a Brauer class are its period and index, this has been widely studied, for example see \cite[4]{AESU}. In this paper, we study some slightly more refined properties of the aforementioned Brauer classes.

We focus on the simplest case: The variety $X$ is a smooth proper curve. The moduli space $M=\mathrm{Pic}^d_{X/k}$ parameterizes degree-$d$ line bundles on $X$. Let $\alpha_d\in\mathrm{Br}(M)$ be the obstruction class. Since $M$ is regular, by \cite[IV.2.6]{milneet}, we know $\alpha_d$ is determined by its restriction to the function field $k(M)$, which we also denote by $\alpha_d$.

The period and index of $\alpha_d$ can be determined in certain cases. When $X$ is the universal genus $g$ curve in characteristic zero, the period $\mathrm{per}(\alpha_d)=(2g-2,d-g+1)$, see \cite[6.4]{melo}. Let $L$ be a tautological line bundle on $(X\times M)_{k^{\mathrm{sep}}}$, then $\alpha_d$ is represented by the Brauer-Severi variety descended from the generic fiber of $\mathbb{P}(\mathrm{pr}_{2,*}(\mathrm{pr}_1^*\omega_X\otimes L)^\vee)$, see \cite[4.9]{giraud}. The Brauer-Severi variety has dimension $(g-2+d)$, so the index $\mathrm{ind}(\alpha_d)$ divides $(g-1+d)$. Since period always divides index, we know $\mathrm{ind}(\alpha_0)=g-1$ and $\mathrm{ind}(\alpha_{g-1})=2g-2$.

In section \ref{sec3}, we study existence of involution on the division algebras:
\begin{theorem*} The division algebra representing $\alpha_0$ has involutions of second kind that extends the natural involution on $\mathrm{Pic}^0_{X/k}$.
\end{theorem*}
\noindent In section \ref{sec4}, we probe the class $\alpha_{g-1}$ at the generic point of theta divisor $\Theta\subset\mathrm{Pic}^{g-1}_{X/k}$:
\begin{theorem*} The Brauer class $\alpha_{g-1}\in\mathrm{Br}(\mathrm{Pic}^{g-1}_{X/k})$ restricts to zero in $\mathrm{Br}(k(\Theta))$.
\end{theorem*}
\noindent In section \ref{sec5}, we show similar result holds for the specialization of $\alpha_0$ to generalized theta divisors $\Theta_E\subset\mathrm{Pic}^0_{X/k}$, where $E$ is a semi-stable vector bundle of rank $2$ and slope $g-1$ on $X$. We will work with universal genus $g$ curves, so that $\Theta_E$ is reduced irreducible.

Interesting questions rise in the study of the class $\alpha_0$. For example, note that the class $\alpha_0$ is represented by an Azumaya algebra away from the identity section, hence it is an unramified Brauer class in $\mathrm{Br}(\mathrm{Pic}^0_{X/k})$, one may ask:
\begin{question*} Does the division algebra of $\alpha_0$ contain an Azumaya order?
\end{question*}
Unramified division algebras without an Azumaya order has been constructed in \cite{Azuorder}. It would be interesting to know if the class $\alpha_0$ also provide such an example.
\vspace{3mm}

\textbf{Acknowledgements.} I am very grateful to my advisor Aise Johan de Jong for his invaluable ideas, enlightening discussions and unceasing encouragement. Thank Max Lieblich, Daniel Krashen and Asher Auel for helpful discussions on the existence of involutions.

\section{Preliminaries}We recall some facts from \cite{BLR}. All curves are assumed to be geometrically integral.
\subsection{The Picard scheme}
Let $k$ be a field, let $X$ be a smooth proper curve defined over $k$.
Consider the relative Picard functor $P'_{X/k}\colon \mathrm{Sch}/k\to\mathrm{Sets},\ T\mapsto \mathrm{Pic}(X\times_kT)/\mathrm{pr}_2^*\mathrm{Pic}(T)$. Let $P_{X/k}$ be the \'etale sheafification of $P'_{X/k}$. The functor $P_{X/k}\colon(\mathrm{Sch}/k)_{\acute{e}t}\to\mathrm{Sets}$ is represented by a group scheme $\mathrm{Pic}_{X/k}=\bigsqcup_{d\in\mathbb{Z}}\mathrm{Pic}^d_{X/k}$. The identity component $\mathrm{Pic}^0_{X/k}$ is an abelian variety, the components $\mathrm{Pic}^d_{X/k}$ are torsors of $\mathrm{Pic}^0_{X/k}$.

The representability of $P_{X/k}$ means there exists a tautological line bundle on $X\times_kU$, where $U$ is some \'etale cover of $\mathrm{Pic}_{X/k}$. The representability of $P'_{X/k}$ means there exists a tautological line bundle on $X\times_k\mathrm{Pic}_{X/k}$ (e.g. this is true when $X$ has a $k$-rational point).

\subsection{The Brauer class}
Consider the Leray spectral sequence associated with the projection $\pi\colon X\times\mathrm{Pic}_{X/k}\to\mathrm{Pic}_{X/k}$, the low-degree terms fit into an exact sequence $$0\to\mathrm{Pic}(\mathrm{Pic}_{X/k})\overset{\pi^*}{\to}\mathrm{Pic}(X\times\mathrm{Pic}_{X/k})\to\mathrm{Mor}(\mathrm{Pic}_{X/k},\mathrm{Pic}_{X/k})\overset{\mathrm{d}_2^{0,1}}{\to} \mathrm{Br}(\mathrm{Pic}_{X/k}).$$ The obstruction to the existence of a tautological line bundle on $X\times \mathrm{Pic}_{X/k}$ is given by a class $$\alpha:=\mathrm{d}_2^{0,1}(1_{\mathrm{Pic}_{X/k}})\in\mathrm{Br}(\mathrm{Pic}_{X/k}).$$
Let $d$ be an integer. We denote the restriction of $\alpha$ to $\mathrm{Pic}^d_{X/k}$ by $\alpha_d$.

\section{Involution}\label{sec3}
Let $k$ be a field, let $G=\mathrm{Gal}(k^{\mathrm{sep}}/k)$. Let $X$ be a proper smooth genus $g$ curve defined over $k$. There exists a natural involution $\sigma\colon\mathrm{Pic}^0_{X/k}\to\mathrm{Pic}^0_{X/k}$ induced by taking dual of line bundles. Let's denote the function field of $\mathrm{Pic}^0_{X/k}$ by $K$, and denote the induced involution on $K$ still by $\sigma$.
The obstruction class $\alpha_0\in\mathrm{Br}(\mathrm{Pic}^0_{X/k})$ is uniquely represented by a division algebra $D$ over $K$. We show $D$ has involutions of second kind extending $\sigma$.

\subsection{Involution of second kind}
Let $E$ be a field with a nontrivial involution $\iota$.
Let $A$ be a central simple algebra defined over $E$.
We say $A$ has an involution of second kind extending $\iota$, if there exists a $\iota$-semilinear ring isomorphism $\widetilde{\iota}\colon A\cong A^{opp}$, such that $\widetilde{\iota}^2=\mathrm{id}_A$.

\begin{lemma}The existence of involutions
 of second kind on a central simple algebra can be examined on any central simple algebra in the same Brauer class.
\end{lemma}
\begin{proof}The existence of involutions on $A$ extending $\sigma$ can be characterized cohomologically by $\mathrm{Cores}_{E/E^\iota}([A])=0\in\mathrm{Br}(E^\iota)$,  see \cite[1.3]{BOI}.
\end{proof}
\subsection{The central simple algebras}
Let's choose suitable central simple algebras that represent the classes $\alpha_0\ \textrm{and}\ \sigma^*(\alpha_0)$. First, let's recall a useful lemma:

\begin{lemma}\label{skn} Let $V,V'$ be $k$-vector spaces of same dimension, then there is a natural isomorphism $\mathrm{Isom}(V,V')/k^*\cong \mathrm{Isom}(\mathrm{End}(V),\mathrm{End}(V'))$, given by $g\mapsto(f\mapsto g\circ f\circ g^{-1}).$
\end{lemma}
\begin{proof}This is the Skolem-Noether lemma, see \cite[2.7.2]{GS} for proof.
\end{proof}

Since $X_{k^{\mathrm{sep}}}$ has rational points, we know tautological line bundles exist on $X_{k^{\mathrm{sep}}}\times(\mathrm{Pic}^0_{X/k})_{k^{\mathrm{sep}}}.$
Let $\eta=\mathrm{Spec}(Kk^{\mathrm{sep}})$ be the generic point of $(\mathrm{Pic}^0_{X/k})_{k^{\mathrm{sep}}}$. Note that different tautological line bundles on $X_{k^{\mathrm{sep}}}\times(\mathrm{Pic}^0_{X/k})_{k^{\mathrm{sep}}}$ differ by pullback of line bundles on $(\mathrm{Pic}^0_{X/k})_{k^{\mathrm{sep}}}$, so all tautological line bundles have the unique (up to isomorphism) restriction on $X_{\eta}$. We denote the unique line bundle on $X_\eta$ by $L$.

Let's denote $Kk^{\mathrm{sep}}$ by $K'$, then\footnote{Take $F=K'\cap k^{\mathrm{sep}}$. If $F\neq k$, then the idempotents in $F\otimes_kF$ would give zero divisors in the function field of $\mathrm{Pic}^0_{X_F/F}$.} $K'\cap k^{\mathrm{sep}}=k$. By \cite[VI.1.12]{Lang}, we know $K'/K$ is a Galois extension with Galois group $G$. Thus we may view $\mathrm{H}^0(X_{K'},\omega_{X_{K'}}\otimes L)$ and $\mathrm{End}(\mathrm{H}^0(X_{K'},\omega_{X_{K'}}\otimes L))$ as $G$-module and $G$-algebra.

\begin{lemma}\label{alg1}
The $G$-algebra $\mathrm{End}(\mathrm{H}^0(X_{K'},\omega_{X_{K'}}\otimes L))$ descends to a central simple algebra over $K$ with Brauer class $\alpha_0$.
\end{lemma}
\begin{proof}
Note that $\mathbb{P}(\mathrm{H}^0(X_{K'},\omega_{X_{k'}}\otimes L)^\vee)$ descends to the a Brauer-Severi variety over $K$ that represents $\alpha_0$, see \cite[V.4.9]{giraud}. By Lemma \ref{skn}, we know $\mathrm{Aut}(\mathbb{P}(V^\vee))=\mathrm{Aut}(\mathrm{End}(V))$. Thus the Brauer-Severi variety represents the same $1$-cocycle as $\mathrm{End}(\mathrm{H}^0(X_{K'},\omega_{X_{K'}}\otimes L))$ in $\mathrm{H}^1(G,\mathrm{PGL}_{g-1}(K'))$.\end{proof}

\begin{lemma}\label{alg2}
The $G$-algebra $\mathrm{End}(\mathrm{H}^0(X_{K'},\omega_{X_{K'}}\otimes L^\vee))$ descends to a central simple algebra over $K$ with Brauer class $\sigma^*(\alpha_0)$.
\end{lemma}
\begin{proof}By the previous lemma, we know $\sigma^*(\alpha_0)$ is represented by $\mathrm{End}(\mathrm{H}^0(X_{K'},\omega_{X_{K'}}\otimes L)\otimes_{K,\sigma}K)$. By \cite[III.9.3]{H}, there exists canonical isomorphism $\mathrm{H}^0(X_{K'},\omega_{X_{K'}}\otimes{L})\otimes_{K,\sigma}K\cong\mathrm{H}^0(X_{K'},\omega_{X_{K'}}\otimes \sigma^*L)$. It suffices to identify $\sigma^*L$ with $L^\vee$ Galois equivariantly.

Note that $\sigma^*L\otimes L$ is the trivial line bundle on $X_{K'}=(X_K)_{K'}$. Take any $K$-section of $\sigma^*L\otimes L$ would yield a $G$-equivariant isomorphism $\sigma^*L\cong L^\vee$.
\end{proof}

\subsection{Characterization by pairing}
We show in our case, existence of involution can be implied by existence of certain pairing.
\begin{lemma}\label{ivo} Let $k$ be a field. Let $D,D'$ be isomorphic central simple $k$-algebras, split by some Galois extension $k'/k$. Let $G:=\mathrm{Gal}(k'/k)$. Let's fix isomorphisms $$\phi\colon D\otimes_k{k'}\cong \mathrm{End}(V),\ \ \phi'\colon D'\otimes_k{k'}\cong\mathrm{End}(V'),$$ where $V,V'$ are $k'$-vector spaces. Then there exist a natural isomorphism $$\mathrm{Isom}(D,D')\cong \mathrm{Isom}(V,V')^G/k^*,\ \ g\mapsto (f\mapsto \phi'^{-1}(g\circ \phi(f)\circ g^{-1})).$$
\end{lemma}
\begin{proof}By Hilbert 90, we know $\mathrm{Isom}(V,V')^G/k^*=(\mathrm{Isom}(V,V')/k^{\prime*})^G$.
Then we conclude from Lemma \ref{skn} and Galois descent.
\end{proof}

\begin{proposition}\label{prop}
If there exists a $\sigma$-equivariant and $G$-equivariant bilinear mapping $$u\colon \mathrm{H}^0(X_{K'},\omega_{X_{K'}}\otimes L)\times \mathrm{H}^0(X_{K'},\omega_{X_{K'}}\otimes L^\vee)\to K',$$
then there exists a $G$-equivariant $\sigma$-linear isomorphism $$\mathrm{H}^0(X_{K'},\omega_{X_{K'}}\otimes L)\to\mathrm{H}^0(X_{K'},\omega_{X_{K'}}\otimes L)^\vee.$$
\end{proposition}
\begin{proof} As in Proposition \ref{alg2}, we take a $G$-invariant isomorphism $\phi\colon L^\vee \to\sigma^*L$. This isomorphism yields a $G$-equivariant $\sigma$-linear isomorphism $$\Phi\colon \mathrm{H}^0(X_{K'},\omega_{X_{K'}}\otimes L)\to \mathrm{H}^0(X_{K'},\omega_{X_{K'}}\otimes L^\vee).$$
Consider the map
$$Q\colon \mathrm{H}^0(X_{K'},\omega_{X_{K'}}\otimes L)\to \mathrm{H}^0(X_{K'},\omega_{X_{K'}}\otimes L)^\vee,$$
$$s\mapsto \left(t\mapsto u(\Phi(t),s)\right).$$
We check $G$-equivariance: For any $g\in G$, it acts on $Q$ as $g\cdot Q\colon t\mapsto g(u(\Phi(g^{-1}t,s))).$
This coincides with $Q$ by the $G$-equivariance of $\Phi$ and $u$. The $\sigma$-linearity follows from the $\sigma$-equivariance of $u$ and $\sigma$-linearity of $\Phi$.
\end{proof}

\subsection{Existence of pairing}
We show the $G$-invariant perfect pairing in the previous section indeed exists.
\begin{lemma}\label{pairing}Let $k$ be an algebraically closed field. Let $V_1,V_2,U$ be $k$-vector spaces. Let $B\colon V_1\times V_2\to U$ be a bilinear pairing with no zero-divisors (i.e. $B(v_1,v_2)\neq 0$ if $v_1\neq 0$ and $v_2\neq0$). If $\mathrm{dim}_kV_1=\mathrm{dim}_kV_2$, then there exists a linear form $\sigma\colon U\to k$ such that $\sigma\circ B\colon V_1\times V_2\to k$ is a perfect pairing.
\end{lemma}
\begin{proof}
Since the bilinear pairing $B$ has no zero-divisors, and $k$ is algebraically closed, the pairing $B$ induces a morphism $\beta\colon\mathbb{P}(V_1)\times\mathbb{P}(V_2)\to\mathbb{P}(U)$. Finding an linear form $\sigma$, up to scaling, is equivalent to pin down a hyperplane $H\subset\mathbb{P}(U)$. The pairing $\sigma\circ B$ is perfect if and only if $H$ does not contain $\beta(t\times \mathbb{P}(V_2))$ for any $t\in\mathbb{P}(V_1)$.

Let $Z=\{(H,t)|H\supset \beta(t\times\mathbb{P}(V_2))\}\subset\mathbb{P}^\vee(U)\times\mathbb{P}(V_1)$ be the incidence subvariety. Denote the two projections on $\mathbb{P}^\vee(U)\times\mathbb{P}(V_1)$ by $\mathrm{pr}_1$ and $\mathrm{pr}_2$. Let $\mathrm{dim}(U)=n,\mathrm{dim}(V_i)=m$. Since $\beta$ restricts to linear embedding in both family of fibers, we know $\mathrm{pr}_2$ is a smooth fibration in $\mathbb{P}^{n-1-m}$, so $\mathrm{dim}(Z)=(m-1)+(n-1-m)=n-2$.
Thus $\mathrm{dim}({\mathrm{pr}_1(Z)})\leq\mathrm{dim}(Z)<\mathrm{dim}(\mathbb{P}^\vee(U))$, hence $\mathbb{P}^\vee(U)\backslash\mathrm{pr_1}(Z)\neq\emptyset$. Any element in the non-empty set gives a hyperplane, whose induced pairing is perfect.
\end{proof}

\begin{theorem} The division algebra at the generic point of $\mathrm{Pic}^0_{X/k}$, which represents the Brauer class $\alpha_0$, has an involution of second kind extending the natural involution on $\mathrm{Pic}^0_{X/k}$.
\end{theorem}
\begin{proof}
Consider the natural bilinear map induced by taking tensor product of sections: $$\mathrm{H}^0(X_{K'},\omega_{X_{K'}}\otimes L)\times \mathrm{H}^0(X_{K'},\omega_{X_{K'}}\otimes L^\vee)\to \mathrm{H}^0(X_{K'},\omega_{X_{K'}}^{\otimes2}).$$
Note that his map is $\sigma$-equivariant, because $\omega_{X_{K'}}=(\omega_X)_{K'}$ is defined over $k$, so $\sigma^*\omega_{X_{K'}}\cong\omega_{X_{K'}}$.
Geometrically, this pairing has no zero-divisors, since on an irreducible curve, product of nonzero rational functions is nonzero. By Lemma \ref{pairing}, geometrically, the perfect pairings form a non-empty Zariski open subset $O\subset \mathbb{P}^\vee(\mathrm{H}^0(X_K,\omega_{X_K}^{\otimes2}))$. Since $K$ is an infinite field, we can always find a $K$-rational point in $O$. This $K$ point yields a desired pairing. Then we conclude by Proposition \ref{prop} and Lemma \ref{ivo}.
\end{proof}
\begin{remark} Let $X/k$ be a smooth genus $g$ curve. Note that there also exists natural involution $\sigma'$ on $\mathrm{Pic}^{g-1}_{X/k}$, induced by $L\mapsto \omega_X\otimes L^\vee$. The same method shows that the class $\alpha_{g-1}$ has a natural involution of second kind extending $\sigma'$.
\end{remark}

\section{Specialization to $\Theta$}\label{sec4}
Let $k$ be a field, let $X$ be a proper smooth genus $g$ curve defined over $k$. We show the class $\alpha\in\mathrm{Br}(\mathrm{Pic}_{X/k})$ splits at the generic point of the theta divisor $\Theta\subset\mathrm{Pic}^{g-1}_{X/k}$.

\subsection{The theta divisor}\label{4.1}
The theta divisor $\Theta\subset\mathrm{Pic}_{X/k}^{g-1}$ is defined as the loci $W_{g-1}^0=\{[M]\in\mathrm{Pic}^{g-1}_{X/k}|h^0(X,M)>0\}$. Recall its scheme structure is defined by the first Fitting ideal of a resolution of $R^1\pi_*L$, where $L$ is a tautological bundle. We briefly recall the construction, see \cite[IV.3]{ACGH1} for details.

Let $U$ be an \'etale cover of $\mathrm{Pic}^{g-1}_{X/k}$, so that a tautological line bundle $L$ exist on $X\times_kU$. Let $\pi\colon X\times_kU\to U$ be the projection. Let's choose an effective canonical divisor $Z\in|\omega_X|$. Let $Z_U=Z\times_k U\subset X\times_k U$. They fit in the diagram:
$$\xymatrix{
&Z_U\ar@{^{(}->}[r] &X\times_k U\ar[rd]^{\pi}\ar[ld]& \\
Z\ar@{^{(}->}[r]&X&         &U.
}$$
Consider the short exact sequence on $X\times_kU$:
$$0\to L(-Z_U)\to L\to L|_{Z_U}\to 0.$$
Take direct image along $\pi_*$, we get exact sequence of sheaves on $U$:
$$0\to\pi_*L\to\pi_*(L|_{Z_U})\xrightarrow{\delta_U} R^1\pi_*(L(-Z_U))\to R^1\pi_*L\to 0.\eqno{(*)}$$
Let's denote $\pi_*{L|_{Z_U}}$ and $R^1\pi_*(L(-Z_U))$ by $F_U$ and $G_U$.
By cohomology and base change \cite[Cor 5.2]{mum}, we know $F_U,G_U$ are locally free of rank $2g-2$, the subsheaf $\pi_*L\subset F_U$ is torsion, so $\pi_*L=0$. Let $\delta_U\colon F_U\to G_U$ be the connecting homomorphism, then $\mathrm{det}(\delta_U)$ induces a nonzero section $s_U\colon\mathcal{O}_U\to\mathrm{det}(F_U)^\vee\otimes\mathrm{det}(G_U)$.
The line bundle $\mathrm{det}(F_U)^\vee\otimes\mathrm{det}(G_U)$ and section $s_U$ descend along the cover $U\to\mathrm{Pic}^{g-1}_{X/k}$. The vanishing locus $T_U$ of $s_U$ descends to a closed subscheme $T\subset\mathrm{Pic}^{g-1}_{X/k}$, this is the theta divisor $\Theta$.

Let $\pi'\colon X\times_k T_U\to T_U$ be the projection, let $L'=L|_{T_U}$, let $Z'=Z\times_kT_U\subset X\times_k T_U$. Similar to $(*)$, we have exact sequence\footnote{The sequence $(**)$ is not the base change of $(*)$ to the closed subscheme $T_U\subset U$, because $L$ is not cohomologically flat.}
$$0\to \pi'_{*}L'\to \pi'_{*}(L'|_{Z'})\overset{\delta'}{\to}R^1\pi'_{*}(L'(-Z'))\to R^1\pi'_{*}L'\to 0.\eqno{(**)}$$
Here is a classical result:

\begin{lemma}\label{rk1}
The coherent sheaf $\pi'_*L'$ is torsion free rank $1$.
\end{lemma}
\begin{proof}The sheaf is torsion free since it is a subsheaf of the locally free sheaf $\pi'_*(L'|_{Z'})$. We check the connecting homomorphism $\delta'$ in $(**)$ has corank $1$: Given a point $t\in T_U$, if the corank of $\delta'|_t$ is at least $2$, then $\det(\delta'_t)\subset \mathfrak{m}_t^2$, thus $t$ lies in non-regular locus of $T_U$. But $T_U$ is reduced, so it is regular at the generic point.
\end{proof}

\subsection{Restriction to $k(\Theta)$}
Let $k(\Theta)$ be the function field of $\Theta$. We show a tautological line bundle exist on $X_{k(\Theta)}$, so $\alpha_{g-1}|_{k(\Theta)}=0$.
We start with a slight generalization of Brauer-Severi varieties.

\begin{lemma}{\label{descend}} Let $X$ be a scheme. Let $h\colon U\to X$ be an \'etale cover of $X$. Let $U^k$ be the $k$-th fiber product of $U$ over $X$. Let $U^2\overset{p_i}{\to} U$, $U^3\overset{p_{ij}}{\to} U^2$ and $U^3\overset{q_i}{\to} U$ be projection maps. Let $F$ be a coherent sheaf on $U$. Assume there exists a line bundle $N\in\mathrm{Pic}(U^2)$, such that there exists isomorphism $\phi\colon p_1^*F\overset{\sim}{\to} p_2^*F\otimes N$ and $\beta\colon p_{12}^*N\otimes p_{23}^*N\overset{\sim}{\to} p_{13}^*N,$ satisfying $(1_{q_3^*F}\otimes\beta)\circ (p_{23}^*\phi\otimes 1_{p_{12}^*N})\circ (p_{12}^*\phi)=p_{13}^*\phi,$ then the scheme $a_{F}\colon P':=\mathbb{P}(F)=\mathrm{Proj}(\mathrm{Sym}^\bullet(F^\vee))\to U$ descends to a scheme $a\colon P\to X.$
\end{lemma}

\begin{proof}Recall in general, given a scheme $S$, a line bundle $L$ on $S$ and a coherent sheaf $F$ on $S$, there exists a unique isomorphism $f\colon A:=\mathbb{P}(F)\overset{\sim}{\to} B:=\mathbb{P}(F\otimes L)$ such that $f^*\mathcal{O}_B(1)\cong\mathcal{O}_A(1)\otimes L$. In our case, the isomorphisms $\phi$ provide the covering datum $\mathbb{P}(\phi)$, the isomorphism $\beta$ provides descent datum for descending $P'$ along $U\to X$. Effectiveness of the descent datum follows from \cite[6.1.7]{BLR}.
\end{proof}

\begin{lemma}\label{Brauerann}
Keep the notation as in the previous lemma. Let $f\colon Y\to X$ be a morphism of schemes. Let $q_{ij}$ be the projection maps from $U\times_XU\times_XY$ to its factors. Let $G$ be a coherent sheaf on $U\times_XY$, such that there exists isomorphism $\psi\colon q_{13}^*G\overset{\sim}{\to} q_{23}^*G\otimes q_{12}^*N,$ satisfying cocycle condition
compatible with $\beta$.
Let $r\colon Y\times_XP'\to P'$ be the projection and let $s\colon Y\times_XP'\to Y\times_XU$ be the base change of $a_{F}$.
Then $s^*G\otimes r^*\mathcal{O}_{P'}(-1)$ descends to a coherent sheaf on $Y\times_XP$.
\end{lemma}
\begin{proof}
Consider the cartesian diagram:
$$\xymatrix{
&Y\times_XP'\ar[rr]^{r}\ar[ld]\ar[dd]^<<<<<s&&P'\ar[dd]\ar[ld]\\
Y\times_XP\ar[dd]\ar[rr]&&P\ar[dd]&\\
&Y\times_XU\ar[ld]\ar[rr]&&U\ar[ld]\\
Y\ar[rr]&&X&
}$$

Let $p_i'\colon P'\times_PP'\to P'$ be the projections. Let $q'_{i3}\colon P'\times_PP'\times_XY\to P'\times_XY$, $q_{12}'\colon P'\times_PP'\times_XY\to P'\times_PP'$ be the projections. Let $t\colon P'\times_PP'\to U\times_XU$ be the product of structure morphisms.

By assumption on $F$, we know $\phi$ induces an isomorphism
$t^*\phi\colon p^{\prime*}_1\mathcal{O}_{P'}(1)\overset{\sim}{\to} p^{\prime*}_2\mathcal{O}_{P'}(1)\otimes t^*N.$ Apply $q^{\prime*}_{12}$ to both sides (note that $r\circ q'_{i3}=p'\circ q'_{12}$), we get an isomorphism $q^{\prime*}_{13}(r^*\mathcal{O}_{P'}(1))\overset{\sim}{\to}q^{\prime'}_{23}(r^*\mathcal{O}_{P'}(1))\otimes q^{\prime*}_{12}t^*N.$
By assumption on $G$, we have isomorphism $t^*\psi\colon q^{\prime*}_{13}s^*G\overset{\sim}{\to}q^{\prime*}_{23}s^*G\otimes q^{\prime*}_{12}t^*N.$ Thus we have the following canonical isomorphism serving as covering datum
\begin{align*}
T_\psi:\hspace{3mm}&q^{\prime*}_{13}(s^*G\otimes r^*\mathcal{O}_{P'}(-1)) \\
\overset{\sim}{\to}\ &q^{\prime*}_{23}s^*G\otimes q^{\prime*}_{12}t^*N\otimes q^{\prime*}_{12}t^*N^{\vee}\otimes q^{\prime*}_{23}(r^*\mathcal{O}_{P'}(-1))=\ q^{\prime*}_{23}(s^*G\otimes r^*\mathcal{O}_{P'}(-1)).
\end{align*}
The cocycle condition for $\psi$ implies the cocycle condition for $T_{\psi}$, hence $s^*\mathcal{G}\otimes r^*\mathcal{O}_{P'}(-1)$ descends to $P\times_XY$.
\end{proof}

We come back to the situation in section \ref{4.1}. Recall we denoted the projection $X\times_kT_U\to T_U$ by $\pi'$, and denoted the restriction of $L$ to $X\times_k T_U$ by $L'$. Let $\widetilde{T}_U=\mathbb{P}(\pi'_*L')$.
By Lemma \ref{descend}, the scheme $\widetilde{T}_U\to T_U$ descends to a scheme $\widetilde{T}\to T$, along the covering $T_U\to T$. We have diagram
$$\xymatrix{
&X\times\widetilde{T}_U\ar[rr]^{r}\ar[ld]\ar[dd]^<<<<<s&&\widetilde{T}_U\ar[dd]\ar[ld]\\
X\times\widetilde{T}\ar[dd]\ar[rr]&&\widetilde{T}\ar[dd]&\\
&X\times T_U\ar[ld]\ar[rr]&&T_U\ar[ld]\\
X\times T\ar[rr]&&T&
}$$

\begin{theorem}
The Brauer class $\alpha_{g-1}\in\mathrm{Br}(\mathrm{Pic}^{g-1}_{X/k})$ restricts to zero in $\mathrm{Br}(k(\Theta))$.
\end{theorem}
\begin{proof} First we show the Brauer class restrict to zero in $\mathrm{Br}(\widetilde{T})$. It suffices to show there exists a tautological line bundle on $X\times_k\widetilde{T}$. We apply Lemma \ref{descend} and Lemma \ref{Brauerann}: Let $h:T_U\to T$ be the covering map, let $F=\pi'_*(L')$ and $G=L'$, then the line bundle $s^*L'\otimes r^*\mathcal{O}_{\widetilde{T}_U}(-1)$ is a tautological line bundle over $X\times_k\widetilde{T}_U$ which descends to $X\times_k\widetilde{T}$. Thus $\alpha_{g-1}$ maps to zero in $\mathrm{Br}(\widetilde{T})$. Then note that $\pi'_{*}(L')$ has generic rank $1$, the projectivization $\widetilde{T}\to T$ is a birational map, so the Brauer class restricts to $0$ in $\mathrm{Br}(k(\Theta))$.
\end{proof}

\begin{remark}
It is not clear if the Brauer class $\alpha_{g-1}$ restricts to zero in $\mathrm{Br}(\Theta)$.
\end{remark}

\section{Specialization to $\Theta_E$}\label{sec5}
Let $X/k$ the universal genus $g$ curve. Let $E$ be a semi-stable vector bundle of rank $2$ and slope $g-1$ on $X$. We show the Brauer class $\alpha_0$ splits at the generic point of generalized theta divisors $\Theta_E\subset\mathrm{Pic}^0_{X/k}$. The proof will be the same as in the previous section, as long as we know $\Theta_E$ is reduced\footnote{This is in general not true, if we work with arbitrary curves, or generalized theta divisor for higher rank vector bundles, see \cite{Hitching}.}.

\subsection{The universal curve}
Let $k_0$ be a fixed field. Let $g\geq3$ be an integer. Let $\mathcal{M}_{g}$ be the moduli stack of families of smooth genus $g$ curves over $k_0$. Let $k$ be the function field of $\mathcal{M}_{g}$. Let $\mathcal{X}\to\mathcal{M}_g$ be the universal family, we call its generic fiber $X/k$ the universal genus $g$ curve (for families of curves over $k_0$).
We collect some facts:
\begin{lemma}{\cite[5.1]{Ste}}\label{schoeer}The group  $\mathrm{Pic}(X)$ is generated by $\omega_X$.\end{lemma}
\begin{lemma}\label{k} For line bundles on $\mathrm{Pic}^d_{X/k}$, the numerical class is always a multiple of $\frac{2g-2}{\mathrm{g.c.d}(2g-2,d+g-1)}\Theta$.\end{lemma}
\begin{proof} This is proved in \cite[1]{kou} when $\mathrm{char}(k_0)=0$. In general, let $P=\mathrm{Pic}^d_{X/k}$, we use the long exact sequence associated to $0\to\mathrm{Pic}^0(P_{k^{\mathrm{sep}}})\to\mathrm{Pic}(P_{k^{\mathrm{sep}}})\to\mathrm{NS}(P_{k^{\mathrm{sep}}})\to 0$.
By \cite[2]{MaxWei}, we know $\mathrm{NS}(P_{k^{\mathrm{sep}}})=\langle\Theta\rangle$. The connecting homomorphism maps $\Theta$ to the torsor $[\mathrm{Pic}^{g-1+d}_{X/k}]\in\mathrm{H^1}(k,\mathrm{Pic}^0_{X/k})$\footnote{One check this explicitly, using autoduality of $\mathrm{Pic}^0_{X/k}$.}. Then note that $[\mathrm{Pic}^1_{X/k}]$ has order $2g-2$, by the strong Franchetta theorem, see \cite{Ste}.
\end{proof}

\subsection{Semi-stable vector bundles on universal curves}
\begin{proposition}There exist semi-stable vector bundles of rank $2$ and slope $g-1$ on $X$.
\end{proposition}
\begin{proof}Let $D\in |\omega_X|$ be an effective canonical divisor, then $D$ is a prime divisor, because $\mathrm{Pic}(X)=\mathbb{Z}\cdot\omega_X$. Hence $D$ is a finite integral scheme over $\mathrm{Spec}(k)$, and $\Gamma(D,\mathcal{O}_D)/k$ is a degree $2g-2$ field extension.

Let $\rho\colon \Gamma(X,\omega_X)\to \Gamma(D,\omega_X|_D)=\Gamma(D,\mathcal{O}_D)$ be the restriction map.
Pick $u\in \Gamma(D,\mathcal{O}_D)\backslash\{0\}$, let $u\rho$ be the multiplication of $\rho$ by $u$. Let $(\rho,u\rho):\omega_X\oplus\omega_X\to \mathcal{O}_D$ be the sum of $\rho$ and $u\rho$. Let $E_u=\mathrm{Ker}(\rho,u\rho)$, then $E_u$ is a locally free rank $2$ sheaf on $X$. It has degree $2g-2$ and slope $g-1$.
We calculate the slope of subsheaves and find suitable $u$ such that $E_u$ is semi-stable.

Let $L$ be a subsheaf of $E_u$. As $X$ is a smooth curve and $E_u$ is locally free, so is $L$.
\begin{enumerate}[leftmargin=1cm]
\item If $\mathrm{rank}(L)=2$, then $\deg(L)=\deg(E_u)-\deg(E_u/L)\leq\deg(E_u)$, so $\mu(L)\leq\mu(E_u)$.
\item If $\mathrm{rank}(L)=1$, since $\mathrm{Pic}(X)=\mathbb{Z}\cdot\omega_X$, we may write $L=\omega_X^{\otimes k}$.
    If $k\leq0$, then $\mu(L)\leq 0\leq\mu(E_u)$.
    If $k\geq 1$, we may pick a nonzero section $t$ of $\omega_X^{\otimes(k-1)}$. The section gives embedding $\omega_X\overset{t}{\hookrightarrow}\omega_X^{\otimes k}\hookrightarrow E_u$, so $h^0(E_u)\geq h^0(\omega_X)=g$. Note that $E_u=\mathrm{Ker}(\rho,u\rho)$, so $h^0(E_u)=2h^0(\omega_X)- \dim_k(\mathrm{Im}(\Gamma(\rho,u\rho)))$, where $\Gamma$ is the functor of taking global section on $X$.
        Thus, in order that $E_u$ is semistable, it suffices to pick $u$ such that $\dim_k(\mathrm{Im}(\Gamma(\rho,u\rho)))\geq g+1$.
\end{enumerate}

Consider the family of coherent sheaves $E_u$ parameterized by $u$, over the vector space $\Gamma(D,\mathcal{O}_D)$. Since rank is a lower semi-continuous function, the condition $\dim_k(\mathrm{Im}(\Gamma(\rho,u\rho)))\geq g+1$ is open for $u\in \Gamma(D,\mathcal{O}_D)-\{0\}\cong k^{2g-2}-\{0\}$. As $k$ is infinite, the $k$-points in $k^{2g-2}$ are Zariski dense. Thus showing the existence of such $u$ is equivalent to showing the open set $\{u\in \mathbb{A}^{2g-2}_k-\{0\}|\dim_k(\mathrm{Im}(\Gamma(\rho,u\rho)))\geq g+1\}$ is nonempty. From now on, it is harmless to assume $k$ is algebraically closed.

Let's choose $D\in|\omega_X|$ such that $D=\sum_{i=1}^{2g-2}P_i$, where $P_i$ are distinct points. Then $$\Gamma(D,\mathcal{O}_{D})=
\oplus_{i=1}^{2g-2} H^0(P_i,\mathcal{O}_{P_i})\cong k^{2g-2}.$$ We show there exists $u=(\lambda_i)_{i=1}^{2g-2}$ with $\lambda_i\in k^*$, such that $\dim_k(\mathrm{Im}(\Gamma(\rho,u\rho)))\geq g+1$.

Let's fix non-vanishing local sections of $\omega_X$ at $P_i$, denoted by $\omega_i$. Then $$\Gamma(\rho,u\rho)\colon \Gamma(X,\omega_X)^{\oplus 2}\to \oplus_{i=1}^{2g-2}\Gamma(P_i,\mathcal{O}_{P_i})$$ is given by $(\Omega_a,\Omega_b)\mapsto \left(f_a(P_i)+\lambda_i f_b(P_i)\right)_i$, where $f_a(P_i)=\frac{\Omega_a(P_i)}{\omega_i(P_i)},f_b(P_i)=\frac{\Omega_b(P_i)}{\omega_i(P_i)}$.

Pick a basis $\{\Omega_i\}_{i=1}^g$ of $\Gamma(X,\omega_X)$. The map $\Gamma(\rho,u\rho)\colon\Gamma(X,\Omega_X)^{\oplus2}\to \oplus_{i=1}^{2g-2} \Gamma(P_i,\mathcal{O}_{P_i})$ can be expressed by the $(2g-2)\times 2g$ matrix $M=[M_L|M_R]$, where $(M_L)_{i,j}=\frac{\Omega_j(P_i)}{\omega_i(P_i)}$ and $(M_R)_{i,j}=\frac{\lambda_i\Omega_j(P_i)}{\omega_i(P_i)}$ for $0\leq i\leq 2g-2,0\leq j\leq g$.

Note that $M_L$ describes the map $\Gamma(\rho)\colon \Gamma(X,\omega_X)\to \oplus_{i=1}^{2g-2}\Gamma(P_i,\mathcal{O}_{P_i})$. This map fits into the exact sequence $0\to k\cong\Gamma(X,\omega_X-D)\to\Gamma(X,\omega_X)\overset{\Gamma(\rho)}{\to} \oplus_{i=1}^{2g-2}\mathrm{H}^0(X,\mathcal{O}_{P_i})$, thus $\mathrm{rank}(M_L)=\dim\mathrm{Im(\Gamma(\rho))}=g-1$. After rearranging the ordering of $P_i$ and replacing $\Omega_i$ by suitable linear combinations, we may assume

\[M_L=
\left[{\begin{array}{cc}
0_{(g-1)\times 1} &I_{g-1} \\
0_{(g-1)\times 1}& A\\
\end{array}}
\right]
\]

Let $L_1=\mathrm{diag}_{i=1}^{g-1}(\lambda_i)$ be the diagonal matrix, let $L_2=\mathrm{diag}_{i=g}^{2g-2}(\lambda_{i})$, let $L_u=\mathrm{diag}(L_1,L_2)$. Then we may write

\[M=
\left[{\begin{array}{cccc}
0_{(g-1)\times 1} &I_{g-1} &0_{(g-1)\times 1} &L_1 \\
0_{(g-1)\times 1}& A &0_{(g-1)\times 1} &L_2A\\
\end{array}}
\right]
\]

Thus $\mathrm{rank}(M)=\dim(\mathrm{Im}(\Gamma(\alpha,u\alpha)))\geq g+1$ if $\mathrm{rank}((L_2^{-1}L_1-I)A)\geq 2$. Let's choose $\lambda_i$ such that $\lambda_i\lambda_{g-1+i}\neq 1$, so that $L_2^{-1}L_1-I$ is invertible. We are done if we know $\mathrm{rank}(A)\geq 2$.

Consider the restriction maps: $$\phi\colon \omega_X\to\omega_X|_{\sum_{i=g}^{2g-2} P_i} \ \ \textrm{and}\ \ \psi\colon\omega_X\to\omega_X|_{\sum_{i=1}^{g-1}P_i}.$$ Our assumption on $M_L$ means $\mathrm{rank}(A)=\mathrm{rank}(\Gamma(\phi))$ and $\mathrm{rank}(\Gamma(\psi))=g-1$.

Rank-nullity formula on $\Gamma(\phi)$ tells us $g-\mathrm{rank}(A)=h^0(\omega_X(-\sum_{i=g}^{2g-2}P_i))$. By Riemann-Roch, this equals to $h^0(\mathcal{O}_X(\sum_{i=g}^{2g-2}P_i))$. On the other hand, rank-nullity formula for $\Gamma(\psi)$ implies that $h^0(\omega_X(-\sum_{i=1}^{g-1}P_i))=1$. Note that this equals to $h^0(\mathcal{O}_X(\sum_{i=g}^{2g-2}P_i))=g-\mathrm{rank}(A)$, as $\omega_X=\sum_{i=1}^{2g-2}P_i$. Since $g\geq 3$, we know $\mathrm{rank}(A)=g-1\geq2$, .\end{proof}

\subsection{The generalized theta divisor}
Let $E$ be a semi-stable sheaf on $X$ obtained as in the last section. Let $U\to\mathrm{Pic}^0_{X/k}$ be an \'etale cover, such that a tautological line bundle $L$ exist on $X\times_k U$. Let $m$ be a positive integer, and $Z\in |m\omega_X|$ be a reduced effective divisor. Let $\pi\colon X\times_k U\to U$ be the projection, let $E_U$ be the constant family. Let $Z_U=Z\times_kU$. Take short exact sequence $$0\to (E_U\otimes L)(-Z_U)\to E_U\otimes L\to (E_U\otimes L)|_{Z_U}\to 0.\eqno{(\dagger)}$$ By cohomology and base change, for $m$ large enough, it induces a long exact sequence, such that the middle two terms are locally free of same rank:
$$
0\to\pi_{*}(E_U\otimes L)\to\pi_*(E_U\otimes L)|_{Z_U}\xrightarrow{\delta_U}R^1\pi_*\left((E_U\otimes L)(-Z_U)\right)\to R^1\pi_*(E_U\otimes L)\to 0.
$$

Let $\eta$ be the generic point of $U$. Since $E$ is semi-stable and $\chi(E)=0$, by \cite[1.6.2]{Ray} we know $h^i(X_{\eta},E_{\eta}\otimes L_{\eta})=0$. Thus $\pi_*(E_U\otimes L)=0$ and $\delta_U$ is injection, its determinant cuts out a divisor in $T_U\subset U$, which descends to the generalized theta divisor $\Theta_E\subset P$.

\subsection{Restriction to $k(\Theta_E)$}
\begin{lemma}
The subscheme $\Theta_E$ is reduced and irreducible.
\end{lemma}
\begin{proof} By \cite[1.8.1]{Ray}, the divisor $\Theta_E$ is numerically equivalent to $2\Theta$. If $\Theta_E$ is reducible or non-reduced, its reduced component will have numerical class $\Theta$, which is not possible, by Lemma \ref{k}.
\end{proof}
\begin{theorem}The class $\alpha_0\in\mathrm{Br}(\mathrm{Pic}^0_{X/k})$ restricts to zero in $\mathrm{Br}(k(\Theta_E))$.
\end{theorem}
\begin{proof} Restrict the exact sequence $(\dagger)$ to $X\times T_U$. Denote the restriction of $E_U, L$ and $Z_U$ by $E',L'$ and $Z'$. Let $\pi'\colon X\times T_U\to T_U$ be the projection. We have exact sequence
$$0\to\pi^{\prime}_{*}(E'\otimes L')\to\pi^{\prime}_*(E'\otimes L')|_{Z'}\xrightarrow{\delta_U}R^1\pi^{\prime}_*\left((E'\otimes L')(-Z')\right)\to R^1\pi^{\prime}_*(E'\otimes L')\to 0.$$

By the same proof in Lemma \ref{rk1}, the reducedness of $\Theta_E$ implies the generic rank of $\pi^{\prime}_{*}(E'\otimes L')$ is $1$. We then conclude as in Theorem \ref{Brauerann}.
\end{proof}

\bibliographystyle{alpha}
\bibliography{references}

\begin{thebibliography}{KMRT98}

\bibitem[ABGV11]{AESU}
Asher Auel, Eric Brussel, Skip Garibaldi, and Uzi Vishne.
\newblock Open problems on central simple algebras.
\newblock {\em Transform. Groups}, 16(1):219--264, 2011.

\bibitem[ACGH85]{ACGH1}
E.~Arbarello, M.~Cornalba, P.~A. Griffiths, and J.~Harris.
\newblock {\em Geometry of algebraic curves. {V}ol. {I}}, volume 267 of {\em
  Grundlehren der Mathematischen Wissenschaften [Fundamental Principles of
  Mathematical Sciences]}.
\newblock Springer-Verlag, New York, 1985.

\bibitem[AW14]{Azuorder}
Benjamin Antieau and Ben Williams.
\newblock Unramified division algebras do not always contain {A}zumaya maximal
  orders.
\newblock {\em Invent. Math.}, 197(1):47--56, 2014.

\bibitem[Bos90]{BLR}
Siegfried Bosch.
\newblock {\em N\'{e}ron Models}.
\newblock Springer Berlin Heidelberg, Berlin, Heidelberg, 1990.

\bibitem[Cal00]{caldararu}
Andrei~Horia Caldararu.
\newblock {\em Derived categories of twisted sheaves on {C}alabi-{Y}au
  manifolds}.
\newblock ProQuest LLC, Ann Arbor, MI, 2000.
\newblock Thesis (Ph.D.)--Cornell University.

\bibitem[Gir71]{giraud}
Jean Giraud.
\newblock {\em Cohomologie non ab\'{e}lienne}.
\newblock Springer-Verlag, Berlin-New York, 1971.
\newblock Die Grundlehren der mathematischen Wissenschaften, Band 179.

\bibitem[GS17]{GS}
Philippe Gille and Tam{\'a}s Szamuely.
\newblock {\em Central simple algebras and Galois cohomology}, volume 165.
\newblock Cambridge University Press, 2017.

\bibitem[Har77]{H}
Robin Hartshorne.
\newblock {\em Algebraic geometry}, volume~52.
\newblock Springer Science \& Business Media, 1977.

\bibitem[HL18]{MaxWei}
Wei {Ho} and Max {Lieblich}.
\newblock {Splitting Brauer classes using the universal Albanese}.
\newblock {\em arXiv e-prints}, page arXiv:1805.12566, May 2018.

\bibitem[HP15]{Hitching}
George~H. Hitching and Christian Pauly.
\newblock Theta divisors of stable vector bundles may be nonreduced.
\newblock {\em Geom. Dedicata}, 177:257--273, 2015.

\bibitem[KMRT98]{BOI}
Max-Albert Knus, AA~Merkurjev, Markus Rost, and Jean-Pierre Tignol.
\newblock The book of involutions. number 44 in american mathematical society
  colloquium publications.
\newblock {\em American Mathematical Society, Providence, RI}, 1998.

\bibitem[Kou91]{kou}
Alexis Kouvidakis.
\newblock The {P}icard group of the universal {P}icard varieties over the
  moduli space of curves.
\newblock {\em J. Differential Geom.}, 34(3):839--850, 1991.

\bibitem[Lan02]{Lang}
Serge Lang.
\newblock {\em Algebra}, volume 211 of {\em Graduate Texts in Mathematics}.
\newblock Springer-Verlag, New York, third edition, 2002.

\bibitem[Mil80]{milneet}
James~S. Milne.
\newblock {\em \'{E}tale cohomology}, volume~33 of {\em Princeton Mathematical
  Series}.
\newblock Princeton University Press, Princeton, N.J., 1980.

\bibitem[Mum08]{mum}
David Mumford.
\newblock {\em Abelian varieties}.
\newblock Published for the Tata Institute of Fundamental Research By Hindustan
  Book Agency International distribution by American Mathematical Society,
  Mumbai New Delhi, 2008.

\bibitem[MV14]{melo}
Margarida Melo and Filippo Viviani.
\newblock The {P}icard group of the compactified universal {J}acobian.
\newblock {\em Doc. Math.}, 19:457--507, 2014.

\bibitem[Ray82]{Ray}
Michel Raynaud.
\newblock Sections des fibr\'es vectoriels sur une courbe.
\newblock {\em Bull. Soc. Math. France}, 110(1):103--125, 1982.

\bibitem[Sch03]{Ste}
Stefan Schr\"oer.
\newblock The strong {F}ranchetta conjecture in arbitrary characteristics.
\newblock {\em Internat. J. Math.}, 14(4):371--396, 2003.

\end{thebibliography}

\end{document}